\documentclass[12pt,a4paper]{article}
\usepackage{amsthm,amsfonts,amsmath,amssymb,bookmark}

\theoremstyle{plain}
\newtheorem{theorem}{Theorem}%[section]
\newtheorem{lemma}[theorem]{Lemma}

\theoremstyle{definition}

\newtheorem{remark}[theorem]{Remark}

\theoremstyle{remark}

\def\bq{\begin{eqnarray}}
\def\eq{\end{eqnarray}}
\def\bqq{\begin{eqnarray*}}
\def\eqq{\end{eqnarray*}}
\def\nn{\nonumber}

\def\eps{\varepsilon}
\def\wto{\rightharpoonup}
\def \ess {\rm ess}

\def\R{\mathbb{R}}

\def\cE {\mathcal{E}}
\def \d {{\rm d}}

\title{Blow-up profile of Bose-Einstein condensate with singular potentials}

\author{Phan Thanh Viet \footnote{T.V. Phan,  Ton Duc Thang university, 19 Nguyen Huu Tho Str., Tan Phong Ward, District 7, Ho Chi Minh City. Email: phanthanhviet@tdt.edu.vn} }

\date{\normalsize\today}

\begin{document}

\maketitle

%%%%%%%%%%%%%%%%%%%%%%%%%%%%%%%%%
%%%%%%%%%%%%%%%%%%%%%%%%%%%%%%%%%

\begin{abstract}
The paper is concerned with the Bose-Einstein condensate described by the attractive Gross-Pitaevskii equation in $\R^2$, where the external potential is unbounded from below. We show that when the interaction strength increases to a critical value, the Gross-Pitaevskii minimizer collapses to one singular point and we analyze the details of the collapse exactly up to the leading order. 
 \bigskip
   
\end{abstract}

%%%%%%%%%%%%%%%%%%%%%%%%%%%%
%%%%%%%%%%%%%%%%%%%%%%%%%%%%

\section{Introduction}

The stability of the Bose-Einstein condensate depends crucially on the interaction between particles. It is remarkable that the condensate may collapse when the interaction is attractive and the number of particles excesses a critical value, see e.g.  \cite{BraSacTolHul-95,SacStoHul-98,KagMurShl-98}. In the present paper, we will consider an example when the details of the collapse can be analyzed exactly up to the leading order. 

We consider the 2D Gross-Pitaevskii energy functional
$$
\cE_a(u)=\int_{\R^2} \Big( |\nabla u(x)|^2 + V(x)|u(x)|^2 -\frac{a}{2} |u(x)|^4\Big) \d x
$$
where $V$ is an external potential $V$ and the interaction is attractive, i.e. $a>0$. The ground state energy of the condensate is given by 
\bq \label{eq:GP}
E(a)= \inf_{u\in H^1(\R^2), \|u\|_{L^2} = 1} \cE_a(u).
\eq
Since $|\nabla u|\ge |\nabla |u||$ pointwise, in consideration of \eqref{eq:GP} we can always restrict to $u\ge 0$. 

The emergence of the Gross-Pitaevskii functional from Schr\"odinger quantum mechanics is well-known, see e.g. \cite{LewNamRou-15} and references therein. In principle, the Gross-Pitaevskii functional describes the energy per particles of a Bose gas of $N$ particles and $a=\lambda N$ where $\lambda$ is the scattering length of the pair interaction between particles. When $\lambda$ is fixed, the increase of $a$ is equivalent to the increase of the number of particles in the Bose gas, as considered in experiments \cite{BraSacTolHul-95,SacStoHul-98,KagMurShl-98}.

When $V=0$, by a simple scaling argument, it is easy to see that $E(a)=-\infty$ if $a>a^*$ and $E(a)=0$ if $a\le a^*$, where $a^*$ is the optimal constant in the Gagliardo-Nirenberg inequality:
\bq 
\label{eq:GN} 
a^* :=\inf_{u\in H^1(\R^2), \|u\|_{L^2} = 1}  \frac{\int_{\R^2} |\nabla u(x)|^2 \d x}{\frac{1}{2}\int_{\R^2}|u(x)|^4 \d x}.
\eq
It is well-known, see e.g. \cite{GidNiNir-81,Weinstein-83,MclSer-87},  that 
\bq \label{eq:a*-Q}
a^*=\int_{\R^2} |Q|^2 = \int_{\R^2}|\nabla Q|^2 = \frac{1}{2} \int_{\R^2}|Q|^4
\eq
where $Q$ is the positive solution to the nonlinear Schr\"odinger equation
\bq \label{eq:Q}
-\Delta Q + Q - Q^3 =0, \quad Q\in H^1(\R^2).
\eq 
Moreover, $Q$ is unique up to translations and it can be chosen to be radially symmetric decreasing. Thus when $V=0$, $E(a)$ has no minimizer if $a<a^*$ and all positive minimizers of $E(a^*)$, i.e. all positive optimizers of \eqref{eq:a*-Q}, are of the form $\beta Q_0(\beta x-x_0)$, where $Q_0=Q/\|Q\|_{L^2}$ and $\beta>0$, $x_0\in \R^2$ can be chosen arbitrary.

Recently, Guo and Seiringer \cite{GuoSei-14} showed that if $V$ is a trapping potential, i.e. $V(x)\ge 0$ and $\lim_{|x|\to \infty}V(x)=\infty$, then $E(a)$ has (at least) a minimizer $u_a$ for all $a<a^*$. Moreover, they prove that if
$$V (x) = h(x) \prod_{j=1}^J |x-x_j|^{p_j}, \quad 0<C^{-1}\le h(x) \le C, $$
then when $a\uparrow a^*$, up to subsequences of $\{u_a\}$, there exists $i_0\in \{1,2,...,J\}$ such that 
$$p_{i_0}=\max \{p_j: 1\le j\le J \} , \quad h(x_{i_0})=\min\{h(x_j):p_j=p_{i_0}\}$$
and  
$$
(a^*-a)^{1/(p_{i_0}+2)} u_a \Big(x_{i_0} + x (a^*-a)^{1/(p_{i_0}+2)}\Big) \to \beta Q_0(\beta x)
$$
strongly in $L^q(\R^2)$ for all $q\in [2,\infty)$, where
$$
\beta =  \left( \frac{p_{i_0} h(x_{i_0})}{2} \int_{\R^2} |x|^p |Q(x)|^2 \d x \right)^{1/(p+2)} .
$$
This result has been extended to ring-shaped trapping potentials \cite{GuoZenZho-16}, flat-well trapping potentials \cite{GuoWangZengZhou-15} and periodic potentials \cite{WanZha-16} (see also \cite{DenGuoLu-15} for a related work with inhomogeneous interactions). 

A common feature in the previous works \cite{GuoSei-14,GuoZenZho-16,GuoWangZengZhou-15,WanZha-16} is that $V\ge 0$ and the condensate collapses at one of the minimizers of $V$. In the present paper, we will consider the case when $V$ is unbounded from below, e.g. $V(x)=-|x|^{-1}$, which corresponds to the gravitational attraction or the Coulombic attraction. 
 This case is interesting because the instability is stronger, i.e. $E(a)\to -\infty$ when $a\uparrow a^*$, and the speed of the collapse is faster.  
\section{Main results}

First, we have a general result on the existence of minimizers. We will denote $V_{\pm}= \max\{\pm V,0\}$.

\begin{theorem} \label{thm1} Assume that $V\in L^1_{\rm loc}(\R^2)$, $\ess \inf V <0$ and $V_-\in L^p(\R^2)+L^q(\R^2)$ with some $1<p<q<\infty$. Then there exists a constant $a_*\in [0,a^*)$ such that for all $a\in [a_*,a^*)$, the minimization problem $E(a)$ in \eqref{eq:GP} has (at least) a minimizer. If $\inf \sigma(-\Delta+V)<0$, we can choose $a_*=0$. 
 \end{theorem}
 
Note that when $a=0$, the existence of minimizers for $E(a)$ reduces to the existence of bound states for $-\Delta+V$. Therefore, the condition $\inf \sigma(-\Delta+V)<0$ emerges naturally. This condition holds if $V\le 0$ (and $\ess \inf V<0$) or $\int V<0$ (and $V$ decays fast enough);  see e.g. \cite{Simon-76}.  

One interesting point in Theorem \ref{thm1} is the effect of the nonlinear interaction term: even if $-\Delta+V$ does not have a bound state, as soon as $\ess \inf V <0$, $E(a)$ still has a minimizer if $a$ is close to $a^*$ sufficiently. The reason is that most of the kinetic energy is canceled by the interaction energy, and hence the condensate is trapped by any small negative well. 

In Theorem \ref{thm1} we do not assume that $V(x)\to \infty$ as $|x| \to \infty$ (but it covers this case as well). Therefore, the compactness of minimizing sequence is not clear as a-priori, and the proof is more difficult than that of the trapping case. 

To describe precisely the blow-up behavior of the minimizers when $a\uparrow a^*$, we will consider the potentials of the form
\bq \label{eq:AS-V}
V(x)= g(x)+ h(x) \sum_{j=1}^J  |x-x_j|^{- p_j}, \quad 0<p_j<2,
\eq
where $0\le g\in L^{\infty}_{\rm loc}(\R^2)$, $h\in L^{\infty}(\R^2)$ such that $h(x_j):=\lim_{x\to x_j} h(x)$ exists and $\min_{1\le j\le J} h(x_j)<0$. Here $J\in \mathbb{N}$ is arbitrary and $\{x_j\}$ are $J$ different points. Since 
$$|x|^{-s} \in L^{1/s+1/2}(\R^2) + L^{4/s} (\R^2),\quad \forall 0<s<2,$$
by Theorem \ref{thm1}, there exists $a_*\in [0,a^*)$ such that $E(a)$ has (at least) a minimizer $u_a$ for all $a\in [a_*,a^*)$. The behavior of $E(a)$ and $u_a$ when $a\uparrow a^*$ is given below.
\begin{theorem} \label{thm2} Let $V$ as in \eqref{eq:AS-V}. Then 
\begin{align*}
 \lim_{a\uparrow a^*}\frac{E(a)}{(a^*-a)^{-p/(2-p)}} = \left( h_0 \int_{\R^2}  \frac{|Q_0(x)|^2}{|x|^{p}} \d x \right)^{2/(2-p)} (a^*)^{p/(2-p)} \left( (p/2)^{2/(2-p)}- (p/2)^{p/(2-p)}\right)
\end{align*}
where $p:=\max\{p_j: h(x_j)<0\}$ and $h_0:= - \min \{h(x_j):p_j=p\}>0.$

Moreover, for every sequence $a_n\uparrow a^*$, there exists a subsequence (still denoted by $a_n$) and $i_{0}\in \{1,2,...,J\}$ such that $p_{i_0}=p$, $h(x_{i_0})=-h_0$ and
$$
\lim_{n\to \infty} (a^*-a_n)^{1/(2-p)} u_{a_n} \Big( x_{i_0} + x(a^*-a_n)^{1/(2-p)} \Big) = \beta Q_0(\beta x)
$$
strongly in $H^1(\R^2)$, where
\bq \label{eq:beta}
\beta= \left( \frac{a^* h_0 p }{2} \int_{\R^2}  \frac{|Q_0(x)|^2}{|x|^{p}} \d x \right)^{1/(2-p)}.
\eq
If the choice of $i_0$ is unique, then we have the convergence for the whole family $\{u_a\}$.  
\end{theorem}

Our method to prove Theorem \ref{thm2} is somewhat different from the approach in the previous works \cite{GuoSei-14,GuoZenZho-16,GuoWangZengZhou-15,WanZha-16}. While the main ideas are similar, our method seems more direct because it is based only on energy estimates. More precisely, in the existing works, the property of $u_a$ is obtained by analyzing the Euler-Lagrange equation associated with the variational problem $E(a)$. In our approach, we will prove that after necessary modifications, $u_a$ converges in $H^1(\R^2)$ to an optimizer of the Gagliardo-Nirenberg inequality \eqref{eq:GN}, and then determine exactly the limit by matching the energy. In fact, we will see from the proof that
\begin{align*}
 \lim_{a\uparrow a^*}\frac{E(a)}{(a^*-a)^{-p/(2-p)}} = \inf_{\lambda>0} \left[ \frac{\lambda^2}{a^*} - \lambda^p h_0 \int_{\R^2}  \frac{|Q_0(x)|^2}{|x|^{p}} \d x \right]
\end{align*}
and $\beta$ in \eqref{eq:beta} is the optimal value $\lambda$ for the right hand side.

We will prove Theorem \ref{thm1} in Section \ref{sec:thm1} and prove Theorem \ref{thm2} in Section \ref{sec:thm2}.

%%%%%%%%%%%%%%%%%%%%%%%%%%%%%%%%%%%%%%%%%
%%%%%%%%%%%%%%%%%%%%%%%%%%%%%%%%%%%%%%%%%

\section{Existence of minimizers} \label{sec:thm1}

In this section we prove Theorem \ref{thm1}. As a preliminary step, we have

\begin{lemma} \label{lem:ea->0} If $V\in L^1_{\rm loc}(\R^2)$, then 
$$\lim_{a\uparrow a^*}E(a) =E(a^*) = \ess \inf V.$$
Moreover, if $V\not\equiv {\rm constant}$, then $E(a^*)$ has no mimimizer. 
\end{lemma}

\begin{proof} As in \cite{GuoSei-14} we consider the trial function
$$
u_\ell(x)=A_\ell \varphi(x-x_0) Q_0(\ell(x-x_0)) \ell
$$
where $x_0\in \R^2$, $0\le \varphi \in C_c^\infty (\R^2)$ with $\varphi(x)=1$ for $|x|\le 1$, and $A_\ell>0$ is a normalizing factor to make $\|u_\ell\|_{L^2}=1$. Since $Q_0,$ $|\nabla Q_0|=O(|x|^{-\frac{1}{2}}e^{-|x|})$ as $|x|\to\infty$ (see \cite[Proposition 4.1]{GidNiNir-81}), we have
\begin{align*}
A_\ell^{-2}&=  \int_{\R^2} \varphi(x-x_0)^2 |Q_0(\ell(x-x_0))|^2 \ell^2 \d x\\
&=\int_{\R^2} \varphi(x/\ell)^2 |Q_0(x)|^2 \d x = 1 + o(\ell^{-\infty})
\end{align*}
and 
\begin{align*}
\int_{\R^2} |\nabla u_\ell|^2 - \frac{a}{2}\int_{\R^2}|u_\ell|^4 &= \ell^2 \left( \int_{\R^2} |\nabla Q_0|^2 - \frac{a}{2}\int_{\R^2}|Q_0|^4  \right) +o(1)_{\ell\to \infty} \\
& =  \frac{\ell^2(a^*-a)}{2}  \int_{\R^2}|Q_0|^4 +o(1)_{\ell \to \infty}.
\end{align*}
Moreover, since $V(x) |\varphi(x-x_0)|^2$ is integrable and $|Q_0(\ell(x-x_0))|^2 \ell^2$ converges weakly to the Dirac-delta function at $x_0$ when $\ell\to \infty$, we have 
\begin{align*}
\int_{\R^2} V |u_\ell|^2  = A_\ell^2  \int_{\R^2} V(x) |\varphi(x-x_0)|^2  |Q_0(\ell(x-x_0))|^2 \ell^2 \d x \to V(x_0) 
\end{align*}
for a.e. $x_0\in \R^2$. Thus in summary,
$$
E(a) \le \cE_a(u_\ell) \le \frac{\ell^2(a^*-a)}{2}  \int_{\R^2}|Q_0|^4 + V(x_0)+ o(1)_{\ell \to \infty} 
$$
for a.e. $x_0\in \R^2$. By choosing $\ell=(a^*-a)^{-1/4}$ and optimizing over $x_0$, we obtain
$$\limsup_{a\uparrow a^*}E(a) \le \ess \inf V.$$
On the other hand, by the Gagliardo-Nirenberg inequality \eqref{eq:GN}, $E(a)\ge E(a^*)\ge \ess\inf V$. Thus
$$ \liminf_{a\uparrow a^*}E(a)=E(a^*)= \ess \inf V.$$
Finally, if $V\not\equiv {\rm constant}$, then by \eqref{eq:GN} again,
$$\cE_{a^*}(u) \ge \int_{\R^2} V|u|^2 > \ess\inf V, \forall \|u\|_{L^2}=1.$$
Thus $E(a^*)$ has no mimimizer. 
\end{proof}

\begin{proof}[Proof of Theorem \ref{thm1}] Let $0\le a<a^*$. Since $V_-\in L^p(\R^2)+L^q(\R^2)$ with $1<p<q<\infty$, by Sobolev's inequality,
$$-\eps \Delta + V \ge -C_\eps, \quad \forall \eps>0.$$
This and \eqref{eq:GN} imply 
$$
\cE_a(u) \ge \left(1-\frac{a}{a^*}-\eps \right) \int_{\R^2} |\nabla u|^2 - C_\eps, \quad \forall \eps>0.
$$
Thus $E(a)>-\infty$. Moreover, if $\{u_n\}$ is a minimizing sequence  for $E(a)$, then it is bounded in $H^1(\R^2)$. By Sobolev's embedding, after passing to a subsequence if necessary, we can assume that $u_n$ converges to a function $u$ weakly in $H^1(\R^2)$ and pointwise.  

We will show that if $E(a)<0$, then $u$ is a minimizer for $E(a)$. Note that by Lemma \ref{lem:ea->0} and assumption $\ess\inf V<0$, there exists $a_*<a^*$ such that $E(a)<0$ for all $a\in (a_*,a^*)$. Moreover, if $\inf \sigma(-\Delta+V)<0$, then by the variational principle, $E(a)\le \inf\sigma(-\Delta+V)<0$ for all $a\in (0,a^*]$. 

Since $\nabla u_n\wto \nabla u$ weakly in $L^2(\R^2)$, we have 
\begin{align*}
\int_{\R^2} |\nabla u_n|^2 =\int_{\R^2} |\nabla u|^2 +  \int_{\R^2} |\nabla (u-u_n)|^2 + o(1)_{n\to \infty}.
\end{align*}
Moreover, since $|u_n|^2\wto |u|^2$ weakly in $L^s(\R^2)$ for all $1<s<\infty$,
$$
\int_{\R^2} V_- |u_n|^2 = \int_{\R^2} V_- |u|^2 + o(1)_{n\to \infty}
$$
On the other hand, from the pointwise convergence $u_n\to u$,  we get
$$
\int_{\R^2} V_+ |u_n|^2 \ge \int_{\R^2} V_+ |u|^2 + o(1)_{n\to \infty}
$$
by Fatou's lemma and
$$
\int_{\R^2} |u_n|^4 =\int_{\R^2} |u|^4 + \int_{\R^2} |u-u_n|^4 +  o(1)_{n\to \infty}
$$
by Brezis-Lieb's refinement of Fatou lemma \cite{BreLie-83}.  In summary, we obtain
\bq \label{eq:existence-k}
\cE_a(u_n) \ge \cE_a(u) +  \int_{\R^2} |\nabla (u-u_n)|^2 -\frac{a}{2} \int_{\R^2} |u-u_n|^4 + o(1)_{n\to \infty}.
\eq
Since $\{u_n\}$ is a mimimizing sequence for $E(a)$ and
$$
 \int_{\R^2} |\nabla (u-u_n)|^2 -\frac{a}{2} \int_{\R^2} |u-u_n|^4 \ge \left(1-\frac{a}{a^*} \right)\int_{\R^2} |\nabla (u-u_n)|^2
$$
by the Gagliardo-Nirenberg inequality \eqref{eq:GN}, we conclude that
\bq \label{eq:existence-k2}
E(a)=\cE_a(u).
\eq

It remains to show that $\|u\|_{L^2}=1$. As a-priori, we have $\|u\|_{L^2}\le 1$ since $\|u_n\|_{L^2}=1$ and $u_n\wto u$ weakly in $L^2(\R^2)$. Moreover, $u\ne 0$ since $E(a)<0$. Thus we can estimate 
\begin{align*}
E(a)=\cE_a(u) = \|u\|_{L^2}^2 \cE_a\left(\frac{u}{\|u\|_{L^2}}\right) + \frac{a}{2} \left(  \|u\|_{L^2}^{-2}-1 \right) \int_{\R^2} |u|^4 
 \ge \|u\|_{L^2}^2 E(a) .
\end{align*}
Since $E(a)<0$, we conclude that $\|u\|_{L^2}=1$. Thus $u$ is a minimizer for $E(a)$.
\end{proof}

\begin{remark} From the above proof, we also obtain the strong convergence of the minimizing sequence in $H^1(\R^2)$. Indeed, since $u_n\wto u$ weakly in $H^1(\R^2)$ and $\|u\|_{L^2}=1$, we obtain $u_n\to u$ strongly in $L^2(\R^2)$. Moreover, from \eqref{eq:existence-k} and \eqref{eq:existence-k2}, we have $\nabla u_n \to \nabla u$ strongly in $L^2(\R^2)$. Thus $u_n\to u$ strongly in $H^1(\R^2)$. 
\end{remark}

%%%%%%%%%%%%%%%%%%%%%%%%%%%%%%%%%%%%%%%%%
%%%%%%%%%%%%%%%%%%%%%%%%%%%%%%%%%%%%%%%%%

\section{Blow-up behavior}\label{sec:thm2}

In this section, we prove Theorem \ref{thm2}. Recall that
$$
V(x)= g(x)+ h(x) \sum_{j=1}^J  |x-x_j|^{- p_j}, \quad 0<p_j<2,
$$
and  
$$p:=\max\{p_j: h(x_j)<0\}, \quad h_0:= - \min \{h(x_j):p_j=p\}>0.$$

Note that $V$ satisfies the conditions in Theorem \ref{thm1} because 
$$|x|^{-s}= |x|^{-s} \chi_{\{|x|\le 1\}}  + |x|^{-s} \chi_{\{|x|\ge 1\} }  \in L^{1/s+1/2}(\R^2) + L^{4/s} (\R^2),\quad \forall 0<s<2.$$
In the following, we always consider the case when $a<a^*$ and $a$ is sufficiently close to $a^*$.  

First, we prove the sharp upper bound on $E(a)$ when $a\uparrow a^*$. 
\begin{lemma} \label{lem:Ea-up} We have
\begin{align*}
\limsup_{a\uparrow a^*}\frac{E(a)}{(a^*-a)^{-p/(2-p)}} \le \inf_{\lambda>0} \left[ \frac{\lambda^2}{a^*} - \lambda^p h_0 \int_{\R^2}  \frac{|Q_0(x)|^2}{|x|^{p}} \d x \right] .
\end{align*}
\end{lemma}

\begin{proof} Without loss of generality, we can assume that $p_1=p$ and $h(x_1)=-h_0<0$. By assumption on $V$,  for every $\eps>0$ there exists $\eta_\eps>0$ such that 
$$
V(x)\le (\eps - h_0) |x-x_1|^{-p}, \quad \forall  |x-x_1|\le 2\eta_\eps.
$$
As in the proof of Lemma \ref{lem:ea->0}, we choose
$$
u_\ell(x)=A_{\ell} \varphi (x-x_1) Q_0(\ell(x-x_1)) \ell
$$
where $0\le \varphi  \in C_c^\infty (\R^2)$ with $\varphi (x)=1$ for $|x|\le \eta_\eps$, $\varphi(x)=0$ for $|x|\ge 2\eta_\eps$, and $A_{\ell}>0$ is a suitable factor to make $\|u_\ell\|_{L^2}=1$. The choice of $\varphi$ ensures that
$$
V(x) \varphi (x-x_1)  \le  (\eps -  h_0) |x-x_1|^{-p} \chi_{\{|x-x_1|\le \eta_\eps\}}.
$$
Since $Q_0,$ $|\nabla Q_0|=O(|x|^{-\frac{1}{2}}e^{-|x|})$ as $|x|\to\infty$ (see \cite[Proposition 4.1]{GidNiNir-81}), when $\ell\to \infty$ independently of $\eps>0$ we have $A_\ell = 1 + o(\ell^{-\infty}),$ 
\begin{align*}
\int_{\R^2} |\nabla u_\ell|^2 - \frac{a}{2}\int_{\R^2}|u_\ell|^4 &= \ell^2 \left( \int_{\R^2} |\nabla Q_0|^2 - \frac{a}{2}\int_{\R^2}|Q_0|^4  \right) +o(1)_{\ell\to \infty} 
\end{align*}
and 
\begin{align*}
\int_{\R^2} V |u_\ell|^2 &\le (\eps - h_0)  A_\ell^2 \int_{|x-x_1|\le \eta_\eps}  \frac{|Q_0(\ell(x-x_1))|^2}{|x-x_1|^{-p}} \ell^2 \d x  \\
&= (\eps - h_0) \ell^{p} A_\ell^2 \int_{|x|\le \ell \eta_\eps}  \frac{|Q_0(x)|^2}{|x|^{p}} \d x \\
&= (\eps - h_0) \ell^{p} \int_{\R^2}  \frac{|Q_0(x)|^2}{|x|^{p}} \d x + o(1)_{\ell\to \infty}.
\end{align*}
Note also that
\bq \label{eq:a*-Q0}
1= \int_{\R^2} |Q_0|^2 = \int_{\R^2}|\nabla Q_0|^2 = \frac{a^*}{2} \int_{\R^2}|Q_0|^4
\eq
because $Q_0=Q/\|Q\|_{L^2}$ and \eqref{eq:a*-Q}. Therefore, in summary, 
\begin{align*}
E(a) \le \cE_a(u_\ell) &\le \ell^2 \left( 1- \frac{a}{a^*} \right) +  (\eps - h_0) \ell^{p} \int_{\R^2}  \frac{|Q_0(x)|^2}{|x|^{p}} \d x + o(1)_{\ell\to \infty} .
\end{align*}
We can choose $\ell=\lambda (a^*-a)^{-1/(2-p)}$ and obtain
\begin{align*}
\frac{E(a)}{(a^*-a)^{-p/(2-p)}} \le  \left[ \frac{\lambda^2}{a^*} + (\eps - h_0) \lambda^p\int_{\R^2}  \frac{|Q_0(x)|^2}{|x|^{p}} \d x \right] + o(1)_{a\to a^*} .
\end{align*}
The desired estimate follows by taking $a\uparrow a^*$, then passing $\eps\to 0$, and finally optimizing over $\lambda>0$. \end{proof}

\begin{lemma}\label{lem:EK} Let $u_a$ be a minimizer for $E(a)$ with $a<a^*$.  When $a$ is sufficiently close to $a^*$, we have
\begin{align*}
-C^{-1}(a^*-a)^{-p/(2-p)} \ge E(a) \ge \int_{\R^2} V |u_a|^2  \ge -C(a^*-a)^{-p/(2-p)}
\end{align*}
and 
$$
 \int_{\R^2}|\nabla u_a|^2 \le C(a^*-a)^{-2/(2-p)} .
$$
\end{lemma}

We always denote by $C\ge 1$ a general constant independent of $a$.  

\begin{proof} From Lemma \ref{lem:Ea-up} and the Gagliardo-Nirenberg inequality \eqref{eq:GN},
\bq \label{eq:super-easy-bound}
-C^{-1}(a^*-a)^{-p/(2-p)} \ge E(a) \ge \int_{\R^2} V |u_a|^2.
\eq
Moreover, by Sobolev's inequality,  $-\Delta-|x|^{-q}\ge -C_q$ for every $0<q<2$. By scaling, we have
$$
s(-\Delta) -|x-y|^{-q} \ge -C_q s^{-q/(2-q)}, \quad \forall y\in \R^2,\forall s>0.
$$
Therefore, by the definition of $V$ and $p=\max\{p_j:h(x_j)<0\}$,
$$
V(x) \ge-C \sum_{j\,:h(x_j)<0} |x-x_j|^{-p_j} -C \ge s \Delta - C s^{-p/(2-p)} -C, \quad \forall s>0.
$$
Using this with $s=(1-a/a^*)/4$ and inequality \eqref{eq:GN} again we obtain
\begin{align*}  %\label{eq:ap-1}
E(a) + \int V |u_a|^2 &=\int_{\R^2} |\nabla u_a|^2 + 2 \int_{\R^2} V |u_a|^2 -\frac{a}{2} \int_{\R^2} |u_a|^4 \nn\\
& \ge \left( 1-\frac{a}{a^*}-2s \right) \int_{\R^2} |\nabla u_a|^2  - C s^{-p/(2-p)} -C \nn\\
&\ge \frac{1}{2}\left( 1-\frac{a}{a^*} \right) \int_{\R^2} |\nabla u_a|^2  - C (a^*-a)^{-p/(2-p)}. 
\end{align*}
Combining with \eqref{eq:super-easy-bound} we conclude 
$$\int_{\R^2} |\nabla u_a|^2 \le C (a^*-a)^{-2/(2-p)}\quad \text{and}\quad \int_{\R^2} V |u_a|^2 \ge - C (a^*-a)^{-p/(2-p)}.$$
\end{proof}

\begin{proof}[Proof of Theorem \ref{thm2}] Let $u_a$ be a minimizer for $E(a)$ with $a<a^*$ and $a$ sufficiently close to $a^*$. As discussed, we can choose $u_a\ge 0$. We will denote
$$
\eps_a :=  (a^*-a)^{1/(2-p)}.
$$
Note that $\eps_a\to 0$ when $a\uparrow a^*$. \\

{\bf Step 1: Extracting the limit.} From Lemma \ref{lem:EK}, we have
$$ 
\eps_a^2 \int_{\R^2} |\nabla u_a|^2 \le C \quad \text{and}\quad \eps_a^{p} \int_{\R^2} V |u_a|^2 \le -C^{-1} .
$$ 
From the latter estimate and the simple bound 
$$
V(x) \ge -C \sum_{h(x_j)<0} |x-x_{i_0}|^{-p_j} - C,
$$
after passing to a subsequence of $u_a$ if necessary, we can find $i_0\in \{1,2,...,J\}$ such that $h(x_{i_0})<0$ and
$$
\eps_a^{p} \int_{\R^2} \frac{|u_a(x)|^2}{ |x-x_j|^{p_{i_0}}} \d x \ge C^{-1}.
$$
Define  
$$
w_a(x):= \eps_a  u_a(x_{i_0}+ \eps_a x).
$$
Then $\|w_a\|_{L^2}=\|u_a\|_{L^2}=1$. Moreover, from the above estimates, we obtain that $w_a$ is bounded in $H^1(\R^2)$ and 
$$
C^{-1} \le \int_{\R^2} \frac{\eps_a^p |u_a(x)|^2}{ |x-x_{i_0}|^{p_{i_0}}} \d x = \eps_a^{p-p_{i_0}} \int_{\R^2}  \frac{|w_a(x)|^2}{ |x|^{p_{i_0}}} \d x .
$$
This implies that $p_{i_0}=p$ (recall $p=\max\{p_j:h(x_j)<0\}$) and
$$
C^{-1} \le  \int_{\R^2}  \frac{|w_a(x)|^2}{ |x|^{p}} \d x .
$$

Since $w_a$ is bounded in $H^1(\R^2)$,  after passing to a subsequence if necessary, we can assume that $w_a$ converges to a function $w$ weakly in $H^1(\R^2)$ and pointwise. Since 
$$ C^{-1} \le \int_{\R^2}  \frac{|w_a(x)|^2}{ |x|^{p}} \d x \to  \int_{\R^2}  \frac{|w(x)|^2}{|x|^{p}} \d x,$$
we have $w\not\equiv 0$. \\

{\bf Step 2: Relating $w$ and $Q_0$.} Next, we prove that $w$ is an optimizer for the Gagliardo-Nirenberg inequality \eqref{eq:GN}. Recall that from Lemma \ref{lem:EK}  we have
$$ -C^{-1} \eps_a^{-p} \ge E(a) \ge \int_{\R^2} V|u_a|^2 \ge -C \eps_a^{-p}.$$
Since $\eps_a\to 0$ and $p<2$, we obtain
\begin{align}
0 &=\lim_{a\uparrow a^*} \eps_a^{2} \Big( E(a) - \int_{\R^2} V|u_a|^2 \Big) \nn\\
&= \lim_{a\uparrow a^*} \eps_a^{2} \left( \int_{\R^2} |\nabla u_a|^2 -\frac{a}{2} \int_{\R^2}|u_a|^4  \right)\nn \\
& = \lim_{a\uparrow a^*} \left(  \int_{\R^2} |\nabla w_a|^2 -\frac{a}{2} \int_{\R^2}|w_a|^4 \right).\label{eq:Ea0}
\end{align}
As in the proof of Theorem \ref{thm1}, since $w_a$ converges to $w$ weakly in $H^1(\R^2)$, 
\begin{align}
&\lim_{a\uparrow a^*} \left( \int_{\R^2} |\nabla w_a|^2 - \int_{\R^2} |\nabla w|^2 -  \int_{\R^2} |\nabla (w_a-w)|^2 \right)=0, \label{eq:Ea1}\\
&\lim_{a\uparrow a^*} \left( \int_{\R^2} |w_a|^4 - \int_{\R^2} |w|^4 - \int_{\R^2} |w_a-w|^4 \right)=0,\label{eq:Ea2}\\
&\lim_{a\uparrow a^*} \left( \int_{\R^2} |w_a|^2 - \int_{\R^2} |w|^2 - \int_{\R^2} |w_a-w|^2\right)=0. \label{eq:Ea3}
\end{align}
Here last two convergences follows from Brezis-Lieb's lemma \cite{BreLie-83}. Combining \eqref{eq:Ea0}, \eqref{eq:Ea1} and \eqref{eq:Ea2} we have 
\begin{align} \label{eq:Ea4}
\lim_{a\uparrow a^*} \left( \int_{\R^2} |\nabla w|^2+ \int_{\R^2} |\nabla (w_a-w)|^2 - \frac{a}{2}\int_{\R^2} |w|^4 - \frac{a}{2}\int_{\R^2} |w_a-w|^4 \right)=0.
\end{align}
On the other hand, by the Gagliardo-Nirenberg inequality \eqref{eq:GN}, 
\begin{align*}
\int_{\R^2} |\nabla w|^2 -\frac{a}{2} \int_{\R^2}|w|^4 &\ge (1- \|w\|_{L^2}^2) \int_{\R^2} |\nabla w|^2.
\end{align*}
Similarly, using \eqref{eq:Ea3} and the fact that $\|\nabla(w_a-w)\|_{L^2}$ is bounded (since $w_a$ is bounded in $H^1(\R^2)$), we have
\begin{align*}
 \int_{\R^2} |\nabla (w_a-w)|^2 - \frac{a}{2} \int_{\R^2} |w-w_a|^4 &\ge (1- \|w_a-w\|_{L^2}^2) \int_{\R^2} |\nabla(w_a-w)|^2 \\
 & = \|w\|_{L^2}^2 \int_{\R^2} |\nabla(w_a-w)|^2 + o(1)_{a\to a^*}.
\end{align*}
Thus \eqref{eq:Ea4} implies that
$$
\limsup_{a\uparrow a^*} \left(  (1- \|w\|_{L^2}^2) \int_{\R^2} |\nabla w|^2 +  \|w\|_{L^2}^2 \int_{\R^2} |\nabla(w_a-w)|^2\right) \le 0.
$$
Since $0<\|w\|_{L^2}\le 1$, we conclude that $\|w\|_{L^2}=1$ and $\|\nabla(w_a-w)\|_{L^2}\to 0$. Thus $w_a\to w$ strongly in $H^1(\R^2)$. The convergence \eqref{eq:Ea0} then implies that $w$ is an optimizer for the Gagliardo-Nirenberg inequality \eqref{eq:GN}. 

Since $Q_0$ is the unique optimizer for \eqref{eq:GN} up to translations and dilations, we conclude that
$$ w(x)=\beta Q_0(\beta x-y_0)$$
for some $\beta>0$ and $y_0\in \R^2$. The values of $\beta$ and $x_0$ will be determined below.\\

{\bf Step 3: Energy lower bound.} Now we derive a sharp lower bound for $E(a)$ when $a\uparrow a^*$. By the definition of $V$, for every $\delta>0$, there exists $C_\delta>0$ such that 
$$
V(x) \ge (h(x_j)-\delta) \sum_{j: h(x_j)<0} |x-x_j|^{-p_j} - C_\delta.
$$
Therefore,
\begin{align} \label{eq:V-lower-bound-0}
 \int_{\R^2} V |u_a|^2 &\ge \sum_{j: h(x_j)<0} (h(x_j)-\delta) \int_{\R^2} \frac{|u_a(x)|^2}{|x-x_j|^{p_j}} \d x -C_\delta \\
& = \sum_{j: h(x_j)<0} (h(x_j)-\delta) \eps_a^{-p_j} \int_{\R^2} \frac{|w_a(x)|^2}{|x-\eps_a^{-1}(x_{j}-x_{i_0})|^{p_j}} \d x -C_\delta.\nn
\end{align}
Here we have used $u_a(x)=\eps_a^{-1} w_a(\eps_a^{-1}(x-x_{i_0}))$ and the change of variables. 

For every fixed $y\in \R^2$, by H\"older's and Sobolev's inequalities, 
\begin{align*}
&\left| \int_{\R^2} \frac{|w_a(x)|^2-|w(x)|^2}{|x-\eps_a^{-1}y|^{p_j}} \d x \right|\\
& \le \left( \int_{\R^2} \frac{|w_a(x)+w(x)|^2}{|x-\eps_a^{-1}y|^{p_j}} \right)^{1/2} \left( \int_{\R^2} \frac{|w_a(x)-w(x)|^2}{|x-\eps_a^{-1}y|^{p_j}} \right)^{1/2} \\
&\le C \|w_a+w\|_{H^1} \|w_a-w\|_{H^1} \to 0. 
\end{align*}
Moreover, if $y\ne 0$, then we can choose $s>0$ such that $(s+1)p_j<2$ and estimate
\begin{align*}
& \int_{\R^2} \frac{|w(x)|^2}{|x-\eps_a^{-1}y|^{p_j}} \d x = \int_{|x| \le \eps_a^{-1}|y|/2} \frac{|w(x)|^2}{|x-\eps_a^{-1}y|^{p_j}} \d x + \int_{|x| > \eps_a^{-1}|y|/2} \frac{|w(x)|^2}{|x-\eps_a^{-1}y|^{p_j}} \d x \\
&\le \int_{\R^2} \frac{|w(x)|^2}{(\eps_a^{-1}|y|/2)^{p_j}} \d x + \left( \int_{\R^2}\frac{|w(x)|^2}{|x-\eps_a^{-1}y|^{(s+1)p_j}} \d x \right)^{1/(s+1)} \left( \int_{|x| > \eps_a^{-1}|y|/2}|w(x)|^2 \d x  \right)^{s/(s+1)} \\
&\le (\frac{2\eps_a}{|y|})^{p_j} + C_s\|w\|_{H^1}^{2/(s+1)}  \left( \int_{|x| > \eps_a^{-1}|y|/2}|w(x)|^2 \d x  \right)^{s/(s+1)} \to 0.
\end{align*}
Thus we have proved that for every fixed $y\in \R^2$, 
$$
\int_{\R^2} \frac{|w_a(x)|^2}{|x-\eps_a^{-1}y|^{p_j}} \d x = \left\{ \begin{aligned}
\int_{\R^2} \frac{|w(x)|^2}{|x|^{p_j}} \d x + o(1)_{a\to a^*}, \quad \text{if } y=0; \\
o(1)_{a\to a^*}, \quad \text{if } y\ne 0.
\end{aligned}
\right.
$$
Thus we deduce from \eqref{eq:V-lower-bound-0} that
\begin{align} \label{eq:V-lower-bound-1}
 \int_{\R^2} V |u_a|^2 \ge (h(x_{i_0})-\delta) \eps_a^{-p}  \int_{\R^2} \frac{|w(x)|^2}{|x|^{p}} \d x - C_\delta + o(\eps_a^{-p}) . 
\end{align}
Here we have also used $p_{i_0}=p$ (which was already proved before). 

Moreover, using the Gagliardo-Nirenberg inequality \eqref{eq:GN} and the convergence $w_a\to w$ in $H^1(\R^2)$  we have
\begin{align}
\int_{\R^2} |\nabla u_a|^2 -\frac{a}{2} \int |u_a|^4 &\ge \left( 1- \frac{a}{a^*} \right) \int_{\R^2} |\nabla u_a|^2 \nn\\
&= \left( 1- \frac{a}{a^*} \right)  \eps_a^{-2} \int_{\R^2}|\nabla w_a|^2 \nn\\
&= (a^*)^{-1} \eps_a^{-p} \Big( \int_{\R^2}|\nabla w|^2 + o(1)_{a\to a^*} \Big). \label{eq:V-lower-bound-2}
\end{align}
Here we have also used  $(a^*-a)\eps_a^{-2}=\eps_a^{-p}$.

Combining \eqref{eq:V-lower-bound-1} and \eqref{eq:V-lower-bound-2}, we obtain
$$
\eps_a^{p} E(a) \ge  (a^*)^{-1}\int_{\R^2}|\nabla w|^2  + (h(x_{i_0})-\delta) \int_{\R^2} \frac{|w(x)|^2}{|x|^{p}} \d x - C_\delta \eps_a^{p}+ o(1)_{a\to a^*} .
$$
We take $a\uparrow a^*$, and then pass $\delta\to 0$. We get
\bq \label{eq:lower-sharp-aa}
\liminf_{a\uparrow a^*} \eps_a^{p} E(a) \ge (a^*)^{-1}\int_{\R^2}|\nabla w|^2 + h(x_{i_0}) \int_{\R^2} \frac{|w(x)|^2}{|x|^{p}} \d x.
\eq

{\bf Step 4: Conclusion.} Now we use $w(x)=\beta Q_0(\beta x -y_0)$. From \eqref{eq:lower-sharp-aa} we have
$$
\liminf_{a\uparrow a^*} \eps_a^{p} E(a) \ge \frac{\beta^2}{a^*} + \beta^p h(x_{i_0}) \int_{\R^2} \frac{|Q_0(x-y_0)|^2}{|x|^{p}} \d x
$$
We have used $\|\nabla Q_0\|_{L^2}=1$; see \eqref{eq:a*-Q0}.

Note that $1>h(x_{i_0})\ge -h_0$. Moreover, by well-known rearrangement inequalities, see e.g.  \cite[Chapter 3]{LieLos-01}, we deduce that $Q_0$ is radially symmetric decreasing and 
$$
\int_{\R^2}\frac{|Q_0(x-y_0)|^2}{|x|^{p}} \d x \le \int_{\R^2}\frac{|Q_0(x)|^2}{|x|^{p}} \d x 
$$
with equality if and only if $y_0=0$. Thus 
\begin{align*}
\liminf_{a\uparrow a^*} \eps_a^{p} E(a) &\ge \frac{\beta^2}{a^*} + \beta^p h(x_{i_0}) \int_{\R^2} \frac{|Q_0(x-y_0)|^2}{|x|^{p}} \d x \\
& \ge \frac{\beta^2}{a^*} - \beta^p h_0 \int_{\R^2} \frac{|Q_0(x)|^2}{|x|^{p}} \d x.
\end{align*}

On the other hand, we have proved in Lemma \ref{lem:Ea-up} that
\begin{align*}
\limsup_{a\uparrow a^*}\frac{E(a)}{(a^*-a)^{-p/(2-p)}} \le \inf_{\lambda>0} \left[ \frac{\lambda^2}{a^*} - \lambda^p h_0 \int_{\R^2}  \frac{|Q_0(x)|^2}{|x|^{p}} \d x \right] .
\end{align*}
Therefore we conclude that
\begin{align*}
\lim_{a\uparrow a^*}\frac{E(a)}{(a^*-a)^{-p/(2-p)}} = \inf_{\lambda>0} \left[ \frac{\lambda^2}{a^*} - \lambda^p h_0 \int_{\R^2}  \frac{|Q_0(x)|^2}{|x|^{p}} \d x \right] ;
\end{align*}
moreover, $h(x_{i_0})=-h_0$, $y_0=0$ and $\beta$ is the optimal value in 
\begin{align*}
&\inf_{\lambda>0} \left[ \frac{\lambda^2}{a^*} - \lambda^p h_0 \int_{\R^2}  \frac{|Q_0(x)|^2}{|x|^{p}} \d x \right] \\
&=\left( h_0 \int_{\R^2}  \frac{|Q_0(x)|^2}{|x|^{p}} \d x \right)^{2/(2-p)} (a^*)^{p/(2-p)} \left( (p/2)^{2/(2-p)}- (p/2)^{p/(2-p)}\right). 
\end{align*}
This means
$$ \beta= \left( \frac{a^* h_0 p }{2} \int_{\R^2}  \frac{|Q_0(x)|^2}{|x|^{p}} \d x \right)^{1/(2-p)}.$$
The proof is finished.
\end{proof}

\end{document}